\documentclass[a4paper]{amsart}

\usepackage[margin=1.4in]{geometry}

\usepackage{graphicx, stmaryrd} 
\usepackage{enumerate}  
\usepackage{hyperref}   
\usepackage[square,numbers]{natbib}     
\usepackage{amsthm}     
\usepackage{amsmath}    
\usepackage{amssymb}    
\usepackage{amsxtra}    
\usepackage{eucal}      
\usepackage[bbgreekl]{mathbbol}   
\usepackage{comment}
\usepackage{tikz}

\usetikzlibrary{matrix}
\usetikzlibrary{patterns, shapes}
\usetikzlibrary{arrows}
\usetikzlibrary{calc,3d}
\usetikzlibrary{decorations,decorations.pathmorphing, decorations.pathreplacing}
\usetikzlibrary{through}
\tikzset{every picture/.style={baseline=-.65ex}}
\tikzset{ext/.style={circle, draw,inner sep=1pt},int/.style={circle,draw,fill,inner sep=1pt},nil/.style={inner sep=1pt}}
\tikzset{exte/.style={circle, draw,inner sep=3pt},inte/.style={circle,draw,fill,inner sep=3pt}}
\tikzset{diagram/.style={matrix of math nodes, row sep=3em, column sep=2.5em, text height=1.5ex, text depth=0.25ex}}
\tikzset{diagram2/.style={matrix of math nodes, row sep=0.5em, column sep=0.5em, text height=1.5ex, text depth=0.25ex}}
\tikzset{every loop/.style={draw}}

\usepackage{tikz-cd}

\theoremstyle{plain}

\newtheorem{thm}{Theorem}[section]

\newtheorem{prop}[thm]{Proposition}

\newtheorem{lemma}[thm]{Lemma}
\newtheorem{cor}[thm]{Corollary}
\theoremstyle{definition}

\newtheorem{ex}[thm]{Example}

\newcommand{\Q}{\mathbb{Q}}

\newcommand{\F}{\mathbb{F}}

\DeclareMathOperator{\vdim}{dim}

\DeclareMathOperator{\ad}{ad}

\newcommand{\Z}{\mathbb{Z}}

\newcommand{\grt}{\mathfrak{grt}_1}
\usepackage{todonotes}

\newcommand{\Lie}{\mathrm{Lie}}

\author{Thomas Willwacher}
\address{Department of Mathematics\\ ETH Zurich\\ R\"amistrasse 101 \\ 8092 Zurich, Switzerland}
\email{thomas.willwacher@math.ethz.ch}

\title{The Deligne--Drinfeld conjecture after OpenAI}

\date{}

\begin{document}

\begin{abstract}
We provide a simplified and more human-friendly version of the proof of the Deligne--Drinfeld conjecture recently given by OpenAI \cite{OpenAI}.
\end{abstract}

\maketitle

\section{Introduction}
Let $K$ be a field and let $\Lie_K(x,y)$ be the free Lie algebra in letters $x$ and $y$ over $K$. It is naturally graded by weight, with $x$ and $y$ carrying weight 1 each, and we denote the homogeneous part of weight $n$ by $\Lie_K(x,y)_n$.

The Grothendieck--Teichmüller Lie algebra $\grt(K)$ is a weight graded Lie algebra whose part $\grt(K)_n\subset \Lie_K(x,y)_n$ of weight $n$ consists of the elements $\psi(x,y)\in \Lie_K(x,y)_n$ satisfying the equations
\begin{align}
\psi(x,y)+\psi(y,x)&=0,\label{eq:antisymmetry}
\qquad \qquad\text{(antisymmetry)}
\\
\psi(x,y)+\psi(y,-x-y)+\psi(-x-y,x)&=0,\label{eq:three-cycle}
\qquad \qquad\text{(3-cycle or hexagon)}
\\
\label{eq:cabling}
[x,\psi(x,y)] - [x+y,\psi(-x-y,y)] &=0, 
\qquad \qquad\text{(cabling)}
\\
\psi(t_{12},t_{23}+t_{24})+\psi(t_{13}+t_{23},t_{34})
 &=\psi(t_{23},t_{34})+\psi(t_{12}+t_{13},t_{24}+t_{34})
 \notag\\
 &\qquad
   +\psi(t_{12},t_{23}), \qquad\text{(pentagon)}
   \label{eq:pentagon}
\end{align}
where the last (pentagon) equation takes values in the Drinfeld--Kohno Lie algebra $\mathfrak t_4(K)$, generated by symbols $t_{ij}=t_{ji}$ for $1\leq i\neq j\leq 4$ with relations 
\begin{equation}\label{eq:braid}
 [t_{ij},t_{k\ell}]=0\quad(i,j,k,\ell\text{ distinct}),\qquad
 [t_{ij},t_{ik}+t_{jk}]=0\quad(i,j,k\text{ distinct}).
\end{equation}

The Lie bracket on $\grt(K)$ is the Ihara bracket
\begin{equation}\label{eq:ihara}
 \{\psi,\varphi\}=[\psi,\varphi]+D_\psi(\varphi)-D_\varphi(\psi),
 \qquad \text{with } D_\psi(x)=0,\quad D_\psi(y)=[y,\psi],
\end{equation}
where $D_\psi$ is extended as a derivation.
We remark that some of the defining equations \eqref{eq:antisymmetry}-\eqref{eq:pentagon} of $\grt(K)$ are actually redundant. For example, if $\mathrm{char}(K)=0$ then Drinfeld \cite[Proposition~5.7]{Drinfeld} (see also \cite[Theorem~5]{BarNatan}) showed that the cabling equation \eqref{eq:cabling} may be omitted, and later work of Furusho \cite[Theorem 3]{Furusho} showed that the pentagon equation \eqref{eq:pentagon} already implies all other equations in weight $n\geq 3$. 

Let 
\[
\mathfrak f(K) = \Lie_K(\sigma_3,\sigma_5,\sigma_7,\dots)
\]
be a free Lie algebra, which we equip with a weight grading by assigning $\sigma_m$ degree $m$.
Let 
\[
d_n := \vdim_K \mathfrak f(K)_n
\]
be the dimension of the weight $n$ part $\mathfrak f(K)_n\subset \mathfrak f(K)$. It is independent of the choice of $K$.
By work of Drinfeld \cite{Drinfeld} there is a map of weight graded Lie algebras 
\begin{equation}\label{equ:DD map}
\mathfrak f(\Q) \to \grt(\Q) 
\end{equation}
sending the generator $\sigma_{2j+1}$ to an element of $\grt(\Q)_{2j+1}$ with nonzero coefficient of $\ad_x^{2j}y$.
The Deligne--Drinfeld conjecture states that this map is an isomorphism.
F. Brown \cite{Brown} already showed that the map is injective for one choice of generators $\sigma_{2j+1}$. By a minor improvement of Brown's argument by the author \cite{WillwacherFreeness}, injectivity in fact holds for any choice of $\sigma_{2j+1}$ with nontrivial $\ad_x^{2j}y$-coefficient.
The remaining part of the Deligne--Drinfeld conjecture was solved by OpenAI \cite{OpenAI}, by showing the following:

\begin{thm}\label{thm:main 1}
    For each $n=1,2,\dots$
    \[
    \vdim_\Q \grt(\Q)_n \leq d_n.
    \]
\end{thm}

In the following we shall provide a streamlined and significantly simplified version of the proof of OpenAI. We also remark that OpenAI \cite{OpenAI} actually shows not only the inequality in Theorem \ref{thm:main 1}, but also provides an independent proof of injectivity of \eqref{equ:DD map}. We shall see in Section \ref{sec:injects} below how this can be realized in our approach as well. This is not strictly necessary for the proof of the Deligne--Drinfeld conjecture, though.

\begin{ex}\label{ex:low}
In weight one \eqref{eq:antisymmetry} forces any $\psi\in \grt(\F_2)$ to have the form $\psi(x,y)=a(x+y)$ with $a\in \F_2$. Inserting into \eqref{eq:pentagon} yields $a(t_{12}+t_{34})=0$, hence $a=0$ and $\grt(\F_2)_1=0$. In weight two necessarily $\psi(x,y)=a [x,y]$ and \eqref{eq:three-cycle} yields $3a [x,y]=0$, hence again $a=0$ and $\grt(\F_2)_2=0$. The first non-trivial weight is weight 3, with $\grt(\F_2)_3$ spanned by $\sigma_3=[x,[x,y]]+[y,[y,x]]$.
\end{ex}

\subsection{Reduction to characteristic 2}
\label{sec:2reduction}
The first reduction is the following.
Note that $\Lie_K(x,y)_n$ has a basis of Lie words, independently of $K$, and likewise for $\mathfrak t_4(K)$, see \cite[\S1.5.2, p.~197]{DarneBraids}. Furthermore, $\grt(K)_n \subset \Lie_K(x,y)_n$ is the set of solutions of a linear system with integral coefficients, independent of $K$.
The rank of an integer matrix over a finite field $\F_p$ is always less than or equal to its rank over $\Q$. Therefore we have 
\[
\vdim \grt(\Q)_n \leq \vdim \grt(\F_p)_n.
\]
Hence Theorem \ref{thm:main 1} is immediately implied by showing that $\vdim_{\F_2} \grt(\F_2)_n \leq d_n$.
We will actually show the following stronger result.
\begin{thm}\label{thm:main 2}
For each $m\geq1$, recursively define the elements
\begin{equation}\label{eq:generator-recursion}
 \sigma_3=[x+y,[x,y]],\qquad
 \sigma_{2m+3}=\{[x,y],\sigma_{2m+1}\}\quad(m\geq1),
\end{equation}
where the braces denote the Ihara bracket.
Then $\sigma_{2m+1}$ is a homogeneous Lie polynomial of weight $2m+1$ belonging to $\grt(\F_2)$. These elements freely generate $\grt(\F_2)$ under the Ihara bracket. In particular,
\[
 \vdim_{\F_2}\grt(\F_2)_n=d_n\qquad(n\geq1).
\]
\end{thm}

The first three generators are
\[
 \sigma_3=[x+y,[x,y]],\qquad
 \sigma_5=[[x,y],\sigma_3], \qquad 
\sigma_7=[[x,y],\sigma_5]+[x+y,[[x,[x,y]],\sigma_3]].
\]
In particular, note that $\sigma_5$ and $\sigma_7$ already do not have a depth 1 term $\ad_x^{2m}y$. The characteristic 2 generators $\sigma_{2m+1}$ should be considered as distinct from the like-named rational generators of $\grt(\Q)$. 
For the remainder of the paper we shall be concerned with showing Theorem~\ref{thm:main 2}.

\subsection*{Acknowledgements}
The author is very grateful to Leila Schneps and Benjamin Enriquez for discussing the problem and sharing their insights on the OpenAI proof. In particular, their reworked version of the OpenAI proof inspired the simplified formula for the representation $\rho$ of Lemma \ref{lem:t4-representation} below, without which the proofs of this paper would have been somewhat lengthier, and the results not as concise.
The author also thanks Dan Petersen for helpful discussions. 

\subsection*{AI usage statement}
The purpose of this paper is to clarify and simplify an AI proof of the Deligne--Drinfeld conjecture. AI tools have also been used extensively in all stages of this work, helping the author in understanding, restructuring and simplifying the original arguments of \cite{OpenAI}. 

\section{The map to \texorpdfstring{$\Lie(B,C)$}{Lie(B,C)}}
\subsection{ABC coordinates}
From this point onwards in the paper we almost always work over the ground field $\F_2$. Where possible we will drop references to the ground field $K$ from notation like $\Lie_K(\cdots)$, and implicitly understand that $K=\F_2$.

Let $\F_2\langle A,B,C\rangle$ be the free associative algebra in letters $A,B,C$. We endow it with a weight grading such that $A$ and $C$ are of weight 2 and $B$ is of weight 1, and we denote by $\F_2\langle A,B,C\rangle_n\subset \F_2\langle A,B,C\rangle$ the part of weight $n$.
Similarly, we denote by $\F_2\langle x,y\rangle$ the free associative algebra in letters $x$ and $y$, both considered of weight $1$.
In each case we understand the free Lie algebras as subsets of the free associative algebras 
\[
\Lie(A,B,C) \subset \F_2\langle A,B,C\rangle, \qquad
\Lie(x,y) \subset \F_2\langle x,y\rangle.
\]

\begin{lemma}\label{lem:abc}
The assignments
\[
 A\longmapsto x^2,\qquad B\longmapsto y,\qquad C\longmapsto[x,y]
\]
define a weight preserving injective map of associative algebras
$$
\iota \colon \F_2\langle A,B,C\rangle \to \F_2\langle x,y\rangle.
$$
The restriction of $\iota$ identifies the Lie ideal $(B,C)\subset\Lie(A,B,C)$ generated by $B$ and $C$ isomorphically with the Lie ideal $(y)\subset\Lie(x,y)$ generated by $y$.
\end{lemma}
\begin{proof}
With the symbol ordering $x>y$ the leading words $xx,xy,y$ form a prefix code. Products
of these words have unique factorizations, proving
injectivity of $\iota$ on $\F_2\langle A,B,C\rangle$, and hence also injectivity of its restriction to $\Lie(A,B,C)$. The elements $e_j=\ad_x^jy$,
$j\ge0$, generate $(y)$. In characteristic two $\ad_x^2=[x^2,-]$ and hence
\[
 e_{2i}
 =\iota(\ad_{A}^i B)
 ,\qquad e_{2i+1}=\iota(\ad_A^iC).
\]
Thus every generator of $(y)$ belongs to the image
$\iota((B,C))$ of the ideal $(B,C)$ and thus $(y)\subset \iota((B,C))$. Conversely,
$\iota(B),\iota(C)\in(y)$, and $\ad_{x^2}=\ad_x^2$ preserves
$(y)$. Hence $\iota((B,C))\subset(y)$, so that $\iota((B,C))= (y)$.
\end{proof}

By Lemma \ref{lem:abc} there is a weight preserving injective map of Lie algebras 
\[
\iota^{-1} \colon (y) \to \Lie(A,B,C).
\]
Since $\grt(\F_2)$ is contained in $(y)$, we may hence consider $\grt(\F_2)$ as a subspace of $\Lie(A,B,C)$ via the embedding $\iota^{-1}$.
In particular, every $\psi\in\grt(\F_2)_n$, $n\geq3$, has
a unique preimage $P=\iota^{-1}\psi\in(B,C)_n$, with
\begin{equation}\label{eq:psi-from-P}
 \psi(x,y)=P(x^2,y,[x,y]).
\end{equation}
In these generators
the symmetry and pentagon relations read
\begin{equation}\label{eq:symmetry-P}
 P(x^2,y,[x,y])=P(y^2,x,[x,y])
\end{equation}
in $\F_2\langle x,y\rangle$, and
\begin{equation}\label{eq:pentagon-naive}
\begin{aligned}
0={}&P\bigl(t_{12}^2,t_{23}+t_{24},[t_{12},t_{23}+t_{24}]\bigr)
+P\bigl((t_{13}+t_{23})^2,t_{34},[t_{13}+t_{23},t_{34}]\bigr)\\
 &+P\bigl(t_{23}^2,t_{34},[t_{23},t_{34}]\bigr)+P\bigl((t_{12}+t_{13})^2,t_{24}+t_{34},
                   [t_{12}+t_{13},t_{24}+t_{34}]\bigr)\\
 &+P\bigl(t_{12}^2,t_{23},[t_{12},t_{23}]\bigr)
\end{aligned}
\end{equation}
in $U(\mathfrak t_4)$. These follow by substituting
\eqref{eq:psi-from-P} into the defining equations.

\subsection{The projection and its image}
We define the weight preserving linear map
\[
 \Phi:\grt(\F_2)\longrightarrow\Lie(B,C),
 \qquad \psi\longmapsto P(0,B,C),\quad P=\iota^{-1}\psi.
\]

\begin{thm}\label{thm:inj}
The map $\Phi$ is injective and its image is contained in the Lie subalgebra
\[
 \Lie(C,[B,C])\subset\Lie(B,C).
\]
\end{thm}
We prove this theorem in Section~\ref{sec:inj proof}. The dimension of the target has a simple description.

\begin{lemma}\label{lem:target}
Writing $g_j=\ad_C^jB$, one has
\begin{equation}\label{eq:target}
 \Lie(C,[B,C])=\F_2C\ltimes\Lie(g_1,g_2,\ldots),
\end{equation}
where the factor $\Lie(g_1,g_2,\ldots)$ is freely generated by the
$g_j$, $j\geq1$. In particular,
$$\vdim\Lie(C,[B,C])_n=d_n \qquad \text{for $n\geq3$}.$$
\end{lemma}
\begin{proof}
Lazard elimination gives
\[
 \Lie(B,C)=\F_2C\ltimes\Lie(g_0,g_1,g_2,\ldots),
\]
with the $g_j$ freely generating the second factor.
Since $[C,g_j]=g_{j+1}$, the subalgebra generated by $C$ and
$g_1=[B,C]$ is exactly the right-hand side of
\eqref{eq:target}. Each $g_j$ has weight $2j+1$, whereas
$C$ has weight $2$, proving the dimension statement.
\end{proof}

\section{Proof of Theorem \ref{thm:inj}}
\label{sec:inj proof}

\subsection{A representation of \texorpdfstring{$\mathfrak t_4$}{t4}}\label{sec:representation}
Let
$R=\F_2\langle A,B,C\rangle$, with $A,C$ of weight $2$ and $B$ of weight $1$.
Write $L_r$ for left multiplication by $r$ and define derivations $E,T$ of $R$ by
\begin{equation}\label{eq:operator-derivations}
\begin{array}{c|ccc}
 &A&B&C\\\hline
 E&[A,B]&0&[A,B]\\
 T&0&[C,B]&[A,B^2].
\end{array}
\end{equation}
Let these operators act on the first factor of
$M=R\otimes\F_2[\varepsilon,\eta]/(\varepsilon^2,\eta^2)$.
Write $\varepsilon,\eta$ also
for multiplication on the second factor. Let
$\partial_\varepsilon$, $\partial_\eta$ be the derivative operators. They are well defined in characteristic two, and
\[
 [\varepsilon,\partial_\varepsilon]=[\eta,\partial_\eta]=1,\qquad \partial_\varepsilon^2=\partial_\eta^2=0,
\]
with all other commutators between $\varepsilon,\eta,\partial_\varepsilon,\partial_\eta$ zero.
Put $G=L_A+L_C+T$.

\begin{lemma}\label{lem:t4-representation}
There is a representation
$\rho:\mathfrak t_4\to\operatorname{End}_{\F_2}(M)$ given by
\begin{equation}\label{eq:t4-representation}
\begin{aligned}
 \rho(t_{12})&=\varepsilon,&
 \rho(t_{23})&=L_A\partial_\varepsilon,\\
 \rho(t_{13})&=\varepsilon+\eta+L_A(\partial_\varepsilon+\partial_\eta),&
 \rho(t_{34})&=L_B+L_C\partial_\eta,\\
\rho(t_{24})&=E+G(\partial_\varepsilon+\partial_\eta)+L_{[A,B]}\partial_\varepsilon\partial_\eta,&
\rho(t_{14})&=\eta+E+L_B+G\partial_\varepsilon+T\partial_\eta+L_{[A,B]}\partial_\varepsilon\partial_\eta.
\end{aligned}
\end{equation}
\end{lemma}
\begin{proof}
First extend the assignments to the free Lie algebra on the
six generators. The rule $[D,L_f]=L_{D(f)}$ for a derivation
$D$ gives
\[
\begin{gathered}
 {}[E,L_A]=[E,L_C]=L_{[A,B]},\qquad [E,L_B]=[T,L_A]=0,\\
 {}[L_A,G]=L_{[A,C]},\qquad [G,L_B]=L_{[A,B]},\\
 {}[G,L_C]+[L_{[A,B]},L_B]=L_{[A,C]}.
\end{gathered}
\]
The last equality uses $[[A,B],B]=[A,B^2]$. These identities yield
\begin{equation}\label{eq:t4-braid-checks}
\begin{gathered}
 \rho([t_{12},t_{13}+t_{23}])=
 \rho([t_{13},t_{12}+t_{23}])=0,\\
 \rho([t_{12},t_{34}])=\rho([t_{13},t_{24}])=0,\\
 \rho([t_{23},t_{24}])=\rho([t_{23},t_{34}])=
\rho([t_{24},t_{34}])=L_{[A,B]}\partial_\varepsilon+L_{[A,C]}\partial_\varepsilon\partial_\eta.
\end{gathered}
\end{equation}
Thus the braid relations hold on strands $123$ and $234$,
as do the two displayed disjointness relations. The six
operators also satisfy
\[
 \sum_{1\leq i<j\leq4}\rho(t_{ij})=0.
\]
Commuting this identity with $\rho(t_{23})$ gives the remaining
disjointness relation. Commuting it with $\rho(t_{12})$ and
$\rho(t_{24})$, respectively $\rho(t_{13})$ and $\rho(t_{34})$,
gives the two independent braid relations on strands $124$,
respectively $134$. Hence all relations \eqref{eq:braid} hold.
\end{proof}

\subsection{Reconstruction and injectivity}
Identify $R\subset M=R\otimes\F_2[\varepsilon,\eta]/(\varepsilon^2,\eta^2)$ with the elements constant in
$\varepsilon,\eta$. 
The following result provides an explicit left inverse to the map $\Phi$ of Theorem \ref{thm:inj}.
\begin{prop}\label{prop:pentagon}
Let $\psi\in\grt(\F_2)_n$, $n\geq3$, and put
$P=\iota^{-1}\psi$ and $f=\Phi(\psi)=P(0,B,C)$. Then
\begin{equation}\label{eq:reconstruction}
 P(A,B,C)=f(E,L_C+T)(1)+f(E+L_B,T)(1),
\end{equation}
with $1\in R$ the unit element. 
In particular, $\Phi$ is injective.
\end{prop}
\begin{proof}
Extend $\rho$ to $U(\mathfrak t_4)$ and apply it to
\eqref{eq:pentagon-naive}. For each of the five pairs $(u,v)$
appearing there, the operators $\rho(u^2),\rho(v),\rho([u,v])$
preserve $R$, with restrictions
\begin{equation}\label{eq:pentagon-restrictions}
\begin{array}{c|ccc}
 (u,v)&\rho(u^2)|_R&\rho(v)|_R&\rho([u,v])|_R\\\hline
 (t_{12},t_{23}+t_{24})&0&E&L_C+T\\
 (t_{13}+t_{23},t_{34})&L_A&L_B&L_C\\
 (t_{23},t_{34})&0&L_B&0\\
 (t_{12}+t_{13},t_{24}+t_{34})&L_A&E+L_B&L_A+T\\
 (t_{12},t_{23})&0&0&L_A.
\end{array}
\end{equation}
These follow directly from \eqref{eq:t4-representation}; for
the squared arguments, use
\[
 \rho(t_{12})^2=\rho(t_{23})^2=0,\qquad
 \rho(t_{13}+t_{23})^2=\rho(t_{12}+t_{13})^2=L_A.
\]
The third and fifth rows contribute zero, since $n\geq3$.
In the fourth row, all terms of $P$ containing $A$ vanish,
since $P$ has no length-one terms and
\[
 [L_A,E+L_B]=[L_A,L_A+T]=0.
\]
Thus the pentagon restricted to $R$ becomes
\begin{equation}\label{eq:full-operator-pentagon}
\begin{aligned}
 L_{P(A,B,C)}&=f(E,L_C+T)+f(E+L_B,L_A+T)\\
 &=f(E,L_C+T)+f(E+L_B,T).
\end{aligned}
\end{equation}
For the second equality we used that $[L_A,E+L_B]=[L_A,T]=0$ and that all terms in $f$ have at least one bracket.
Applying \eqref{eq:full-operator-pentagon} to $1$ proves
\eqref{eq:reconstruction}. 
It determines $P$, hence $\psi$,
from $f$, proving injectivity.
\end{proof}

\subsection{The image of \texorpdfstring{$\Phi$}{Phi} and proof of Theorem \ref{thm:inj}}\label{sec:Phi image}
Keep $\psi,P,f$ as in Proposition~\ref{prop:pentagon}.
To prove Theorem \ref{thm:inj} it remains to show that $f(B,C)\in \Lie(C,[B,C])$. 
On the algebra $\F_2\langle B,C\rangle$ define a derivation $\partial$ by
\begin{equation}\label{eq:partial}
 \partial(B)=C,\qquad \partial(C)=0.
\end{equation}
Define an algebra morphism 
\begin{gather*}
    \theta \colon R=\F_2\langle A,B,C\rangle \longrightarrow 
    \operatorname{End}_{\F_2}\bigl( \F_2\langle B,C\rangle \bigr) 
    \\
     \theta(A)=L_{B^2},\qquad
 \theta(B)=\partial,\qquad
 \theta(C)=L_C,
\end{gather*}
with $L_C$ again denoting left multiplication by $C$.
We next want to apply $\theta$ to the reconstruction formula \eqref{eq:reconstruction}. To this end, we need to understand how $\theta$ intertwines with the derivations $E,T$ of \eqref{eq:operator-derivations} appearing there.
First note that $\theta(T(A))=\theta(T(B))=\theta(T(C))=0$, and hence $\theta(T(r))=0$ for any $r\in R$.
Next, define on $\F_2\langle B,C\rangle$ the derivation $d$ by
\[
 d(B)=B^2+C,\quad d(C)=[B,C].
\]
Then by checking on the generators $A,B,C$ one sees that for any $r\in R$
\[
\theta(E(r))=[d,\theta(r)].
\]
Hence the linear map $Q:R\to \F_2\langle B,C\rangle$, $Q(r)=\theta(r)(1)$,
satisfies $QE=dQ$ and $QT=0$. Moreover, for $r,s\in R$,
\[
 Q(rs)=\theta(r)\theta(s)(1)=\theta(r)Q(s),
 \qquad\text{so}\qquad QL_r=\theta(r)Q.
\]
Now apply $Q$ to \eqref{eq:reconstruction}. Using the intertwining formulas above we obtain the following equations in $\F_2\langle B,C\rangle$,
\[
\begin{aligned}
 Q (f(E,L_C+T)(1))&=f(d,L_C)( Q(1))= f(d,L_C)(1),\\
 Q(f(E+L_B,T)(1))&=f(d+\partial,0)(Q(1))=0.
\end{aligned}
\]
Hence from \eqref{eq:reconstruction} we obtain
\begin{equation}\label{eq:QP1}
 Q(P)=f(d,L_C)(1).
\end{equation}
Use the symmetry equation \eqref{eq:symmetry-P} with the substitutions $x=\partial$ and $y=L_B$ to obtain, with $\partial^2=0$ and $[\partial,L_B]=L_C$,
\begin{equation}\label{eq:thetaP}
 \theta(P)=P(L_{B^2},\partial,L_C)
          =P(0,L_B,L_C)=L_f.
\end{equation}
Applying this to $1\in \F_2\langle B,C\rangle$ yields $Q(P)=f$, and hence with \eqref{eq:QP1}
\begin{equation}\label{eq:image-fixed-point}
 f=f(d,L_C)(1).
\end{equation}
Now one checks that $d^2=0$ so that also $\ad_d^2=0$. Hence the ideal generated by $L_C$ in
$\Lie(d,L_C)$ is the Lie subalgebra generated by
$L_C$ and $[d,L_C]=L_{[B,C]}$.
Evaluation at $1$ maps this Lie subalgebra into
$\Lie(C,[B,C])$. Thus $f\in \Lie(C,[B,C])$ as desired, and Theorem \ref{thm:inj} follows. \hfill\qed

\section{The Ihara bracket and free generators}
Note that Theorem \ref{thm:inj} shown above already implies Theorem \ref{thm:main 1} via Lemma \ref{lem:target}, and hence also the Deligne--Drinfeld conjecture.
In this section we shall go further and study the compatibility of the map $\Phi$ with the Ihara bracket \eqref{eq:ihara}, to show Theorem \ref{thm:main 2}. The results of this section are not needed for the proof of the Deligne--Drinfeld conjecture.

\subsection{Closedness under the Ihara bracket}
It is well known that $\grt(\Q)$ is closed under the Ihara bracket. The proof given in \cite{BarNatan} can be adapted to finite characteristic as well. However, we did not find a direct citable reference in the literature, unfortunately. So for completeness, we provide a proof for the case of characteristic $2$ relevant to this paper. In fact, in characteristic 2 the following slightly stronger result is true.

\begin{prop}\label{prop:closure}
    Let $\psi,\varphi\in
    \F_2\cdot[x,y]\oplus\grt(\F_2)$. Then
    \[
    \{\psi,\varphi\}\in\grt(\F_2),
    \]
    where $\{-,-\}$ denotes the Ihara bracket \eqref{eq:ihara} and $\F_2\cdot[x,y]$ the 1-dimensional span of $[x,y]\in \Lie(x,y)$.
\end{prop}
\begin{proof}
For a Lie polynomial $\rho\in \Lie(x,y)$, put
\[
C_\rho(x,y)=[x,\rho(x,y)]+[x+y,\rho(x+y,y)],
\]
so that the cabling equation \eqref{eq:cabling} for $\rho$ becomes $C_\rho(x,y)=0$. 
In the following it will also be convenient to use the usual notation 
\[
\rho_{I,J,K} := \rho\bigl(\sum_{i\in I, j\in J} t_{ij}, \sum_{j\in J, k\in K} t_{jk}\bigr)
\]
for disjoint $I,J,K\subset \{1,2,3,4\}$, for example 
\begin{align*}
    \rho_{1,2,3}&:= \rho(t_{12},t_{23}), &
    \rho_{1,23,4}&:= \rho(t_{12}+t_{13},t_{24}+t_{34}).
\end{align*}
For a solution $\rho\in \Lie(x,y)$ of the pentagon \eqref{eq:pentagon} and symmetry \eqref{eq:antisymmetry} equations of weight at least two,
consider the assignment
\begin{equation}\label{eq:delta as}
\begin{aligned}
\delta_\rho(t_{12})&=0,\\
\delta_\rho(t_{23})=\delta_\rho(t_{13})
 &=[t_{23},\rho(t_{12},t_{23})]=[t_{23},\rho_{1,2,3}],\\
\delta_\rho(t_{34})
 &=[t_{34},\rho(t_{13}+t_{23},t_{34})]=[t_{34},\rho_{12,3,4}],\\
\delta_\rho(t_{24})
 &=[t_{24}+t_{34},\rho(t_{12},t_{23}) 
 +\rho(t_{12}+t_{13},t_{24}+t_{34})]
 +[t_{34},\rho(t_{13}+t_{23},t_{34})]
 \\&=[t_{24}+t_{34},\rho_{1,2,3} +\rho_{1,23,4}]
 +
 [t_{34},\rho_{12,3,4}],
\end{aligned}
\end{equation}
with $\delta_\rho(\sum_{i<j}t_{ij})=0$ determining its value
on $t_{14}$.
Modulo the central element $\sum_{i<j}t_{ij}$, the presentation
\eqref{eq:braid} is equivalent to the six relations
\[
\begin{gathered}
\relax[t_{12},t_{34}]=[t_{12},t_{13}+t_{23}]
 =[t_{23},t_{12}+t_{13}]=0,\\
[t_{23},t_{24}+t_{34}]=[t_{34},t_{23}+t_{24}]
 =[t_{13},t_{24}]=0.
\end{gathered}
\]
The assignment \eqref{eq:delta as} preserves the first five relations. For the last one, the pentagon, also with strands $2,3$
exchanged, yields
\begin{equation}\label{eq:closure-relation-error}
\begin{aligned}
\delta_\rho[t_{13},t_{24}]
 &=[t_{24},[t_{23},\rho_{1,2,3}]]
 +[t_{13},[t_{24}+t_{34},\rho_{1,2,3}+\rho_{1,23,4}]]
 +[t_{13},[t_{34},\rho_{12,3,4}]]
 \\&=[t_{24},[t_{23},\rho_{1,2,3}]]
 +[t_{13},[t_{24},\rho_{1,2,3}+\rho_{1,23,4}]]
 +[t_{13},[t_{34},\rho_{1,2,34}+\rho_{2,3,4}]]
 \\&=[t_{24},[t_{12},\rho_{1,2,3}]]
 +[t_{13},[t_{24},\rho_{1,23,4}]]
 +[t_{13},[t_{34},\rho_{2,3,4}]]
  \\&=[t_{24},[t_{12},\rho_{1,2,3}]]
 +[t_{13},[t_{24},\rho_{1,3,2}+\rho_{1,3,24}+\rho_{13,2,4}+\rho_{3,2,4}]]
 +[t_{13},[t_{34},\rho_{2,3,4}]]
   \\&=[t_{24},[t_{12},\rho_{1,2,3}]]
 +[t_{13},[t_{24},\rho_{1,3,2}+\rho_{3,2,4}]]
 +[t_{13},[t_{34},\rho_{2,3,4}]]
   \\&=[t_{24},[t_{12},\rho_{1,2,3}] +[t_{13},\rho_{1,3,2}] ]
 +[t_{13},[t_{34},\rho_{2,3,4}]+ [t_{24},\rho_{3,2,4}]]
 \\&=[t_{24},C_\rho(t_{12},t_{23})]
  +[t_{13},C_\rho(t_{34},t_{23})].
\end{aligned}
\end{equation}
In particular, $\delta_\rho$ is a derivation if
$\rho\in\grt(\F_2)$.
This also holds for $\rho=[x,y]$: its pentagon defect is
$[t_{12},t_{34}]+[t_{13},t_{24}]=0$, and
\[
\begin{aligned}
C_{[x,y]}(t_{12},t_{23})
 &=\operatorname{ad}_{t_{23}}^2(t_{13}),
 &
C_{[x,y]}(t_{34},t_{23})
 &=\operatorname{ad}_{t_{23}}^2(t_{24}).
\end{aligned}
\]
Hence \eqref{eq:closure-relation-error} vanishes, since
$\operatorname{ad}_{t_{23}}^2$ is a derivation in characteristic two
and $[t_{13},t_{24}]=0$.
By linearity, $\delta_\psi$ and $\delta_\varphi$ are therefore
derivations of $\mathfrak t_4$.

Denote the two sides of
the pentagon \eqref{eq:pentagon} by
\[
\begin{aligned}
L(\rho)&=\rho_{1,2,34}
 +\rho_{12,3,4}, &
R(\rho)&=\rho_{1,2,3}
 +\rho_{1,23,4}
 +\rho_{2,3,4}.
\end{aligned}
\]
Then for $\theta:= \{\psi,\varphi\}=[\psi,\varphi]+D_\psi\varphi+D_\varphi \psi$ we want to show that $L(\theta)=R(\theta)$, where $D_{\psi}$ is the Ihara derivation as in \eqref{eq:ihara}. We first compute the various derivation terms appearing in the pentagon equation.
The definition of $\delta_\psi$ immediately gives
\[
(D_\psi\varphi)_{1,2,3}
 =\delta_\psi(\varphi_{1,2,3}),
\qquad
(D_\psi\varphi)_{12,3,4}
 =\delta_\psi(\varphi_{12,3,4}).
\]
Moreover, the derivation $\delta_\psi+\ad_{\psi_{1,2,3}}$ of $\mathfrak t_4$ sends
$t_{12}+t_{13}$ to zero and $t_{24}+t_{34}$ to
$[t_{24}+t_{34},\psi_{1,23,4}]$. Hence
\[
(D_\psi\varphi)_{1,23,4}
 =\delta_\psi(\varphi_{1,23,4})
  +[\psi_{1,2,3},\varphi_{1,23,4}].
\]

Next, $\delta_\psi+\ad_{\psi_{12,3,4}}$ annihilates $t_{12}$,
since $t_{12}$ commutes with both $t_{13}+t_{23}$ and $t_{34}$.
Furthermore,
\[
\begin{aligned}
&(\delta_\psi+\ad_{\psi_{12,3,4}})(t_{23}+t_{24})\\
&\quad=[t_{23}+t_{24}+t_{34},
            \psi_{1,2,3}+\psi_{12,3,4}]
       +[t_{24}+t_{34},\psi_{1,23,4}]\\
&\quad= [t_{23}+t_{24}+t_{34},\psi_{1,2,34}+\psi_{2,3,4}+ \psi_{1,23,4}]
+[t_{24}+t_{34},\psi_{1,23,4}]\\
&\quad=[t_{23}+t_{24},\psi_{1,2,34}]
       +[t_{34},\psi_{1,2,34}]
       +[t_{23}+t_{24}+t_{34},\psi_{2,3,4}]
       +[t_{23},\psi_{1,23,4}]\\
&\quad=[t_{23}+t_{24},\psi_{1,2,34}].
\end{aligned}
\]
Here the second equality uses the pentagon, and the final three
terms vanish by the relations on $\mathfrak t_4$. Thus
\[
(D_\psi\varphi)_{1,2,34}
 =\delta_\psi(\varphi_{1,2,34})
  +[\psi_{12,3,4},\varphi_{1,2,34}].
\]

Finally, $\delta_\psi+
\ad_{\psi_{1,2,3}+\psi_{1,23,4}}$ sends $t_{23}$ to
\[
[t_{23},\psi_{1,23,4}]=0
\]
and sends $t_{34}$ to
\[
\begin{aligned}
&[t_{34},\psi_{12,3,4}+\psi_{1,2,3}+\psi_{1,23,4}]=[t_{34},\psi_{1,2,34}+\psi_{2,3,4}]
 =[t_{34},\psi_{2,3,4}],
\end{aligned}
\]
using the pentagon and $[t_{34},\psi_{1,2,34}]=0$. Consequently,
\[
(D_\psi\varphi)_{2,3,4}
 =\delta_\psi(\varphi_{2,3,4})
  +[\psi_{1,2,3}+\psi_{1,23,4},\varphi_{2,3,4}].
\]
The same identities hold with $\psi$ and $\varphi$ interchanged. Summing the two evaluations occurring in $L$ now gives
\[
\begin{aligned}
L(\theta)
={}&\delta_\psi L(\varphi)+\delta_\varphi L(\psi)\\
 &+[\psi_{1,2,34},\varphi_{1,2,34}]
   +[\psi_{12,3,4},\varphi_{12,3,4}]\\
 &+[\psi_{12,3,4},\varphi_{1,2,34}]
   +[\varphi_{12,3,4},\psi_{1,2,34}]\\
={}&\delta_\psi L(\varphi)+\delta_\varphi L(\psi)
   +[L(\psi),L(\varphi)].
\end{aligned}
\]
Likewise, summing the three evaluations occurring in $R$ gives
\[
R(\theta)
 =\delta_\psi R(\varphi)+\delta_\varphi R(\psi)
   +[R(\psi),R(\varphi)].
\]
Since $L(\psi)=R(\psi)$ and $L(\varphi)=R(\varphi)$,
we conclude that $L(\theta)=R(\theta)$.
This also implies symmetry: substituting
\[
t_{12}=t_{34}=x,\qquad t_{23}=t_{14}=y,\qquad
t_{13}=t_{24}=x+y
\]
into the pentagon equation yields $\theta(x,y)+\theta(y,x)=0$.

It remains to prove the cabling relation $C_\theta=0$. 
Temporarily work in $\F_2\langle x,y\rangle$ and introduce the involution $\tau: x\mapsto (x+y), y\mapsto y$. Then cabling reads $C_\rho = (1+\tau) [x,\rho(x,y)]=0$.
For $\rho \in \F_2\cdot[x,y]\oplus \grt(\F_2)$ let $\lambda_\rho$ be the coefficient of $[x,y]$. 
Let 
\[
a_\rho = (1+\tau)\rho + \lambda_\rho y^2.
\]
Then using $C_\rho=\lambda_\rho [y^2,x]$ we may show that 
\[
\tau D_\rho \tau = D_\rho + \ad_{a_\rho}
\]
by checking that the derivations on both sides agree on $x$ and $y$.
We also compute by expanding both sides
\begin{align*}
\theta+\tau\theta
=D_\psi a_\varphi+D_\varphi a_\psi+[a_\psi,a_\varphi].
\end{align*}

Furthermore, we know that $D_\theta=[D_\psi, D_\varphi]$ so that by the previous two equations
\begin{align*}
\tau D_\theta\tau &= \tau [D_\psi, D_\varphi] \tau 
=
[\tau D_\psi\tau , \tau D_\varphi \tau] =
[D_\psi + \ad_{a_\psi},D_\varphi + \ad_{a_\varphi}]
\\&=
D_\theta + \ad_{D_\psi a_\varphi + D_\varphi a_\psi + [a_\psi,a_\varphi]}
=
D_\theta + \ad_{\theta +\tau \theta}.
\end{align*}
Evaluating this at $x+y$ yields the desired cabling relation:
\[
0 = \tau D_\theta\tau(x+y) = D_\theta (x+y) + [x+y,\theta +\tau \theta]
=
[y,\theta(x,y)] + [x+y,\theta(x,y) + \theta(x+y,y)] =C_\theta.  
\]

Finally, symmetry and cabling imply the three-cycle relation
in weights at least two: Cabling applied to $(x,y)$ and
$(x,x+y)$ shows that
\[
\theta(x,y)+\theta(y,x+y)+\theta(x+y,x)
\]
commutes with $x$, and symmetry shows that it also commutes with $y$. It therefore vanishes, since $\Lie(x,y)$ has zero
center. Thus $\theta\in\grt(\F_2)$.
\end{proof}

\subsection{An extension of \texorpdfstring{$\Phi$}{Phi} and compatibility}
By Proposition \ref{prop:closure} we may define the one-dimensional extension of $\grt(\F_2)$,
\[
\grt^{ex}(\F_2):=\F_2\cdot [x,y]\ltimes \grt(\F_2) \subset \Lie(x,y),
\]
equipped with the Ihara bracket. Our morphism $\Phi$ extends to a morphism
\begin{gather*}
\Phi^{ex} \colon\grt^{ex}(\F_2) \to \Lie(B,C) \\
\Phi^{ex}(\psi) = (\iota^{-1}\psi)(0,B,C),
\end{gather*}
so that $\Phi^{ex}\mid_{\grt(\F_2)}=\Phi$.
We have $\Phi^{ex}([x,y])=C$, so Theorem \ref{thm:inj} implies that this extended version of $\Phi$ is still injective and takes values in $\Lie(C,[B,C])$.

\begin{lemma}\label{lem:PhiLie}
    The map $\Phi^{ex}\colon \grt^{ex}(\F_2) \to \Lie(B,C)$ is a morphism of weight graded Lie algebras, i.e., for any $\psi,\varphi\in\grt^{ex}(\F_2)$ we have 
    \[
    \Phi^{ex}(\{\psi, \varphi\})= [\Phi^{ex}(\psi),\Phi^{ex}(\varphi)].
    \]
\end{lemma}
\begin{proof}
Define an algebra morphism
\[
\chi\colon\F_2\langle x,y\rangle\longrightarrow
\operatorname{End}_{\F_2}\bigl(\F_2\langle B,C\rangle\bigr),
\qquad \chi(x)=\partial,\quad\chi(y)=L_B,
\]
with $\partial$ the derivation \eqref{eq:partial}.
For $\psi=\iota(P)\in\grt^{ex}(\F_2)$ and $f=\Phi^{ex}(\psi)$, the identities
$\partial^2=0$ and $[\partial,L_B]=L_C$ then imply
\[
\chi(\psi)=
\psi(\partial, L_B)
=
P(\partial^2, L_B, L_C)
=
P(0, L_B, L_C)
=
L_f,
\]
see also  \eqref{eq:thetaP}.
Since $\partial$ annihilates $C$ and $[B,C]$,
Theorem~\ref{thm:inj} gives $\partial(f)=0$. Thus
\begin{equation}\label{eq:chiD agree}
\chi (D_\psi(x))=0=[L_f,\partial],\qquad
\chi(D_\psi(y))=\chi([y,\psi(x,y)])=[L_B,L_f].
\end{equation}
We obtain that for any $r\in\F_2\langle x,y\rangle$
\[
\chi (D_\psi(r))=[L_f,\chi(r)]
\]
since both sides satisfy the Leibniz rule relative to $\chi$ and agree on $r=x,y$ by \eqref{eq:chiD agree}.
For $\varphi\in\grt^{ex}(\F_2)$ and $g=\Phi^{ex}(\varphi)$, applying $\chi$ to
the Ihara bracket yields
\begin{equation*}
L_{\Phi^{ex}(\{\psi,\varphi\})}
= 
\chi(\{\psi,\varphi\})
=
\chi([\psi,\varphi] +D_\psi\varphi + D_\varphi(\psi) )
=[L_f,L_g]+[L_f,L_g]+[L_g,L_f]=L_{[f,g]}.
\end{equation*}
Hence $\Phi^{ex}(\{\psi,\varphi\})=[f,g]$, so that $\Phi^{ex}$ is indeed a morphism of Lie algebras.
\end{proof}

\subsection{Proof of Theorem \ref{thm:main 2}}
From Lemma \ref{lem:PhiLie} we immediately obtain:
\begin{cor}\label{cor:main 3}
    The morphism 
    \[
    \Phi^{ex} \colon \grt^{ex}(\F_2) \to \Lie(C,[B,C])
    \]
    is an isomorphism of weight graded Lie algebras.
\end{cor}
\begin{proof}
    By Theorem \ref{thm:inj} the morphism is injective and takes values in $\Lie(C,[B,C])$. Since $\Phi^{ex}([x,y])=C$ and $\Phi^{ex}(\sigma_3)=\Phi(\sigma_3)=[B,C]$ both generators are in the image and the result follows.
\end{proof}

\begin{proof}[Proof of Theorem \ref{thm:main 2}]
    Restricting Corollary \ref{cor:main 3} to weights $\geq 3$ we see that $\Phi\colon \grt(\F_2)\to \Lie(C,[B,C])_{\geq 3}$ is an isomorphism of weight graded Lie algebras.
    By Lemma \ref{lem:target} we know that $\Lie(C,[B,C])_{\geq 3}$ is a free Lie algebra in generators $g_m=\ad_C^mB$, $m=1,2,\dots$. 
    We will show by induction that $\Phi(\sigma_{2m+1})=g_m$, thus proving Theorem \ref{thm:main 2}.
    Indeed, $\Phi(\sigma_3)=[B,C]$ and using Corollary \ref{cor:main 3} and the induction hypothesis we moreover have 
    \begin{align*}
    \Phi(\sigma_{2m+3})
    &=
    \Phi^{ex}(\sigma_{2m+3})
    =
    \Phi^{ex}(\{[x,y],\sigma_{2m+1}\})
    \\&=
    [\Phi^{ex}([x,y]), \Phi^{ex}(\sigma_{2m+1}) ]
    =
    [C,g_m] = g_{m+1}
    \end{align*}
    as desired.
\end{proof}

\subsection{The lower bound and independent proof of injectivity}
\label{sec:injects}
We note that the original OpenAI proof \cite{OpenAI} of the Deligne--Drinfeld conjecture is self-contained and does not rely on the injectivity of \eqref{equ:DD map} proven by Brown \cite{Brown}. It is also possible to give a self-contained proof of injectivity of \eqref{equ:DD map} from Theorem \ref{thm:main 2} as follows.

First, let $\Z_{(2)}=\{\frac{p}{q} \mid \text{$q$ odd}\}\subset \Q$ be the localization of $\Z$ at the prime 2. Define $L:=\grt(\Q)\cap \Lie_{\Z_{(2)}}(x,y)$ and let $L_n$ be its weight $n$ component.
Then $L_n$ is free since $\Z_{(2)}$ is a PID and $L_n$ is a submodule of the free module $\Lie_{\Z_{(2)}}(x,y)$. A $\Z_{(2)}$-basis of $L_n$ is also a $\Q$-basis of $\grt(\Q)_n$ and descends to an $\F_2$-basis of $L_n/2L_n$, so that all three modules have the same rank.
The same is also true for the weight $n$ parts of the Chevalley--Eilenberg complexes of the Lie algebras $\grt(\Q)$, $L$, $L/2L$. Since the $\Q$-ranks of the differentials in the Chevalley--Eilenberg complex are at least as large as the $\F_2$-ranks, we have in particular 
\begin{equation}\label{eq:H2grt}
\vdim_{\Q} H_2(\grt(\Q))_n \leq \vdim_{\F_2} H_2(L/2L)_n,
\end{equation}
where $H_2(-)$ denotes the Chevalley--Eilenberg (Lie algebra) homology. 
Now $L/2L$ injects into $\grt(\F_2)$. Since the latter is free by Theorem \ref{thm:main 2}, so is $L/2L$ by the Shirshov--Witt Theorem \cite[Theorem~2.5]{Reutenauer}. Hence $H_2(L/2L)=0$ and by \eqref{eq:H2grt} also $H_2(\grt(\Q))=0$. A positively graded Lie algebra with vanishing second homology
is free, see \cite[Corollary~5.7]{HainTorelli}, and thus $\grt(\Q)$ is free.

Now the images of the generators $\sigma_m$ under \eqref{equ:DD map} have nonzero depth one terms, which cannot appear in any Lie (Ihara) bracket in $\grt(\Q)$. The classes of the $\sigma_{2j+1}$ in the abelianization of $\grt(\Q)$ are therefore nonzero and, having distinct weights, linearly independent. Since $\grt(\Q)$ is free, \cite[Theorem~2.9]{Reutenauer} implies that these elements
freely generate a Lie subalgebra. This proves the injectivity of \eqref{equ:DD map}, and together with Theorem~\ref{thm:main 1} and the argument of Section \ref{sec:2reduction} that \eqref{equ:DD map} is an isomorphism.

\bibliographystyle{alpha}
\bibliography{refs}

\end{document}